\documentclass{amsart}  

\usepackage[colorinlistoftodos,bordercolor=orange,backgroundcolor=orange!20,linecolor=orange,textsize=scriptsize]{todonotes}
\presetkeys%
    {todonotes}%
    {inline}{}

\usepackage{graphics} 
\usepackage{amsmath} 
\usepackage{amssymb}  
\usepackage{mathtools}
\usepackage{csquotes}
\usepackage{frenchineq}

\usepackage{pgfplots, pgfplotstable}
\usepgfplotslibrary{statistics}

\usepackage{cleveref}

\newif\iftrop
\troptrue

\newif\ifintro
\introtrue

\newtheorem{theorem}{Theorem}

\newtheorem{corollary}{Corollary}
\newtheorem{remark}{Remark}

\newcommand{\A}{\mathcal{A}}

\newcommand{\G}{\mathcal{G}}

\newcommand{\W}{\mathcal{W}}

\renewcommand{\P}{P}
\newcommand{\Q}{Q}
\newcommand{\Qs}{\mathcal{Q}}
\renewcommand{\S}{S}
\newcommand{\T}{T}
\newcommand{\Tsel}{T^{\text{sel}}}

\newcommand{\R}{\mathbb{R}}
\newcommand{\Sn}{\mathcal{S}_n}
\newcommand{\Snp}{\mathcal{S}_n^+}
\newcommand{\Snpp}{\mathcal{S}_{np}^+}

\newcommand{\D}{\mathcal{D}}
\newcommand{\M}{\Sigma}
\newcommand{\U}{\mathcal{U}}

\newcommand{\ub}[1]{#1^{\uparrow}}
\newcommand{\mub}{\bigvee}
\DeclareMathOperator{\ricc}{ricc}
\DeclareMathOperator{\trace}{trace}
\DeclareMathOperator{\diag}{diag}
\DeclareMathOperator{\tr}{tr}

\newcommand{\rhot}{\widehat{\rho}}

\title{
Tropical Kraus maps for optimal control of switched systems
}

\author{St\'ephane Gaubert and Nikolas Stott
}

\begin{document}



\begin{abstract}
Kraus maps (completely positive trace preserving maps) arise classically in quantum information, as they describe the evolution of non-commutative probability measures. We introduce tropical analogues of Kraus maps, obtained by replacing the addition of positive semidefinite matrices by a multivalued supremum with respect to the L\"owner order. We show that non-linear eigenvectors of tropical Kraus maps determine piece-wise quadratic approximations of the value functions of switched optimal control problems. This leads to a new approximation method, which we illustrate by two applications: 1) approximating the joint spectral radius, 2) computing approximate solutions of Hamilton-Jacobi PDE arising from a class of switched linear quadratic problems studied previously by McEneaney. We report numerical experiments, indicating a major improvement in terms of scalability by comparison with earlier numerical schemes, owing to the "LMI-free" nature of our method.
\end{abstract}
\thanks{The authors were partially supported by the ANR projects {CAFEIN, ANR-12-INSE-0007} and MALTHY, ANR-13-INSE-0003, by ICODE and by the academic research chair ``Complex Systems Engineering'' of \'Ecole polytechnique - THALES - FX - DGA - DASSAULT AVIATION - DCNS Research - ENSTA ParisTech - T{\'e}l{\'e}com ParisTech - Fondation ParisTech - FDO ENSTA and by the PGMO programme of EDF and FMJH.}
\thanks{St\'ephane Gaubert and Nikolas Stott are with INRIA and CMAP, \'Ecole polytechnique, CNRS, France,        {\tt\small firstname.lastname@inria.fr}}%

\maketitle

\section{INTRODUCTION}

\subsection{Curse of dimensionality attenuation methods}

Dynamic programming is one of the main methods to solve optimal control problems. It characterizes the value function as the solution of a functional equation
or of a Hamilton-Jacobi partial differential equation. It provides a feedback law that is guaranteed to be globally optimal. However, it is subject to the ``curse of dimensionality''. Indeed, the main numerical methods,
including monotone finite difference
or semi-Lagrangean schemes~\cite{crandall-lions,capuzzodolcetta,falcone-ferretti,carlini-falcone-ferretti},
and the anti-diffusive schemes~\cite{zidani-bokanowski},
are grid-based.
It follows that the time needed to obtain an approximate solution with a given accuracy is exponential in the dimension of the state space.

Recently, two numerical methods have been shown, theoretically or practically, to attenuate the curse of dimensionality, for specific classes of optimal control problems with switches.

McEneaney considered hybrid optimal control problems in which a discrete
control allows one to switch between different linear quadratic
models. The method he developed~\cite{mceneaney07}
approximates the value function by a supremum
of elementary functions like quadratic forms, hence
it belongs to the family of ``max-plus basis methods''~\cite{a5,a6}. 
The method of~\cite{mceneaney07} has a remarkable feature:
it attenuates the curse of dimensionality, as shown by the complexity
estimates of Kluberg and McEneaney~\cite{mccomplex} and of Qu~\cite{qusico}.
McEaneney's method~\cite{mceneaney07} has been studied and extended in a series of works~\cite{eneaneyphys,qucdc,kaisemceneaney,mceneaneydower}. 

A different problem consists in 
computing the joint spectral radius of a finite set of matrices~\cite{Jun09}.
This can be formulated as an ergodic optimal control problem
for a switched system. In this
context, the ergodic value function is known 
as the Barabanov norm. Specific numerical
methods have been developed, which approximate the Barabanov ball by a polytope~\cite{Guglielmi2014}, or are of semi-Lagrangean type~\cite{Kozyakin10}.
Ahmadi et al.~\cite{pathcomplete} developed
a new method, based on a path complete automaton.
It approximates the Barabanov norm by a supremum of quadratic norms. 
Whereas the worst case complexity
estimates in~\cite{pathcomplete} are still subject to a curse of dimensionality,  in practice, the efficiency of the method is determined
by the complexity of the optimal switching law rather than 
by the dimension itself.
It allows one to solve instances of dimension inaccessible by grid-based method.

The method of McEneaney~\cite{mceneaney07}, like the one of 
Ahmadi et al.~\cite{pathcomplete},
provide max-plus basis expansions
of approximate value functions. In both methods, semidefinite programming (solution of LMI, linear matrix inequalities) is the bottleneck. Indeed, LMI arise in the ``pruning step'', i.e., the elimination of redundant quadratic forms,
an essential ingredient of McEaneney's method. It was observed
in~\cite{qucdc} that 99\% of the CPU time was spent in the solution
of semidefinite programs. The method of Ahmadi et al.\ involves
a truncation method, considering switching sequences of a fixed length,
and it requires
the solution of a semidefinite program whose size
is exponential in this length.






\subsection{Contribution}
In this paper, we introduce a new approximation method for optimal control
of switched systems. This method still relies on the approximation
of the value function by a supremum of quadratic forms. However, it avoids
the recourse to LMI. We exploit the geometric properties
of the space of positive semidefinite matrices equipped
with the L\"owner order. A key ingredient is the introduction
of the tropical analogues of the Kraus maps arising in quantum information
and control~\cite{sepulchre}. 
The latter 
are quantum Markov operators, describing
the evolution of density matrices (the quantum analogues
of probability measures). 
They act on the space of positive semidefinite matrices, and can be written as $T(X)=\sum_i A_i XA_i^\dag$, where $(\cdot)^\dag$ denotes the adjoint
of a matrix. Tropical Kraus maps are defined
by replacing the sum in the definition of $T(X)$ by a multivalued
operator, providing the set of minimal upper bounds with respect to the L\"owner order. 
Tropical Kraus maps may be thought of as ``1-player''
non-commutative dynamic programming operators
(classical Kraus maps correspond to the 0-player case).
We show that every non-linear eigenvector of the tropical Kraus
map yields an approximation of the value function. Moreover,
the non-linear eigenvalue yields an upper bound for the joint
spectral radius. We show that a non-linear eigenvector
does exist. We compute non-linear eigenvectors through
an iterative scheme,
in which at each stage, a specific minimal upper bound
of a collection of matrices is selected. The latter operation is implemented
in an algebraic way, leading to a fast algorithm (not relying on semidefinite
programming).  We report numerical experiments, showing a major speedup,
allowing  us to treat instances of dimension inaccessible by earlier dynamic programming methods.


The paper is organized as follows. In Section~\ref{sec-2}, we recall
the definitions of the joint spectral radius and switched linear quadratic 
control problems, which will serve as benchmarks. In Section~\ref{sec:mub},
we recall some basic properties of the L\"owner order on the
space of positive semidefinite matrices.  In Section~\ref{sec-kraus},
we introduce tropical Kraus maps, establish the existence
of non-linear eigenvectors, and present the iterative scheme.
In Section~\ref{sec-exp}, we present numerical experiments.






\section{CLASSES OF SWITCHED SYSTEMS}
\label{sec-2}
In this section, we describe the two optimal control problems to which
we will apply our method.
\subsection{Joint spectral radius}
\label{sec:clas_jsr}

We consider here the stability under arbitrary switching of discrete-time linear switched systems
as studied in e.g. \cite{Liberzon03,SunGe11,branicky1998multiple}. 

Let $ \A = \{ A_1, \dots, A_m \}$ be a set of real $n \times n$ matrices.
A discrete-time switched linear system is described by: 
\begin{align*}
x_{k+1} = A_{\sigma(k)} x_k , \; \sigma(k) \in\{ 1, \dots, m \}
\end{align*}
where $x_k \in \R^n$ denotes the trajectory of the system, and $\sigma$ 
is the switching mechanism, which selects one of the matrices
in $\A$ at each instant.

We are interested in the approximation of the {\em joint spectral radius} \cite{Jun09} 
associated to this set. The latter is defined by
\begin{align*}
\rho(\A)  &\coloneqq \lim_{k \rightarrow +\infty} \max_{1\leq i_1,\dots,i_k\leq m}
\| A_{i_1}\dots A_{i_k}\|^{1/k} \enspace .
\end{align*}

A fundamental result of Barabanov~\cite{Barabanov88}
shows that if $\A$ is irreducible,
meaning that there is no nontrivial subspace of $\R^n$ that is left
invariant by every matrix in $\A$, then there is a norm $v$ on $\R^n$ such
that 
\[
\lambda v(x) = \max_{1\leq i\leq m} v(A_i x) ,\qquad \forall x\in \R^n  \enspace ,
\]
for some $\lambda>0$. 
The scalar $\lambda$ is unique and it coincides with the joint spectral radius
$\rho(\A)$.
This shows that, when $\A$ is irreducible, all the trajectories of the switched linear system converge to zero if and only if $\rho < 1$.

The norm $v$ is known as a \emph{Barabanov norm}.
A norm which satisfies the inequality $\max_i v(A_i x) \leq \rho(\A) v(x)$ for all $x\in \R^n$ is called an \emph{extremal norm}.

Extremal norms and Barabanov norms are generally non unique and  cannot be computed exactly,
except in special examples.
Hence, we shall be content
with an {\em approximate extremal norm} $v$, i.e., a solution
of 
\[  \max_{1\leq i\leq m} v(A_i x) \leq \mu v(x) , \qquad \forall x\in \R^n \,,
\]
where $\mu> 0 $. Then, it is readily seen that $\mu \geq \rho(\A)$,
so that an approximate extremal norm yields a guaranteed approximation
of the joint spectral radius.

For instance, an approximate extremal norm can be obtained via an \emph{ellipsoidal norm}, as described in~\cite{ACM}, meaning that there is some positive definite matrix $Q$ and a real $\mu$ such that
\begin{align*}
A_i^\top Q A_i \preceq \mu^2 Q \;\forall i \,,
\end{align*}
where $\preceq$ denotes the L\"owner order (see~\Cref{sec:mub}) and $(\cdot)^\top$ denotes the transpose.
The approximation is usually coarse, since $1 \leq \mu/\rho(\A) \leq \sqrt{n}$ is a tight estimate.
This approach can be refined by \emph{lifting} the set of matrices $\A$ into another set $\A'$ which has the same joint spectral radius, but where the quadratic form $Q$ gives a better approximation of an extremal norm~\cite{pathcomplete}. We adopt such an approach in this paper, see~\Cref{rem:trop_rem} in~\Cref{sec:trop_kraus_def}.

\subsection{Linear Quadratic Optimal Control Problems with Switches}
\label{sec:clas_hjb}

We also consider the following problem of optimal switching between linear quadratic models, studied by McEneaney~\cite{mceneaney07}, namely
approximating the value function $V$ of an optimal control problem having both a control $u$ taking values in $\R^p$ and a discrete control (switches between different modes) $\mu$ taking values in $\Sigma \coloneqq \{ 1, \dots, m \}$:
\begin{align*}
V(x) = \sup_{u \in \U} \sup_{\mu \in \D} \sup_{t > 0} \int_0^t \frac{1}{2} \xi(s)^\top  D^{\mu(s)} \xi(s) - \frac{\gamma^2}{2} |u(s)|^2 \,ds \,.
\end{align*}
Here, $\D$ denotes the set of measurable functions from $[0,+\infty)$ to $\M$ (i.e.\ switching functions), $\U \coloneqq L^2([0,+\infty), \R^p)$ is 
the space of $\R^p$-valued control functions, 
and the state $\xi$ is subject 
to
\begin{align*}
\dot{\xi}(s) = A^\sigma \xi(s) + B^\sigma u(s) \,,\quad \xi(0) = x \,,
\end{align*}
where $\sigma=\mu(s)$ denotes the mode that is selected at time $s$.

It is known~\cite{mceneaney07} that, under some assumptions on the parameters, the value function $V$ takes finite values and is the unique viscosity solution of the stationary Hamilton-Jacobi-Bellman PDE:
\begin{align*}
H(x,\nabla V) = 0\,, \quad x \in \R^n 
\enspace .
\end{align*}
The Hamiltonian $H(x,p)$ in the latter equation is the point-wise maximum of simpler Hamiltonians $H^\sigma(x,p)$ given for $\sigma \in \Sigma$ by
\begin{align*}
H^\sigma(x,p) = ( A^\sigma x)^\top p + \frac{1}{2} x^\top D^\sigma x + \frac{1}{2} p^\top Q^\sigma p \,,
\end{align*}
and $Q^\sigma =\gamma^{-2} B^\sigma ( B^\sigma)^\top$.


We associate with this problem the \emph{Lax-Oleinik} semi-group $\{S_t\}_{t \geq 0}$ defined by
\begin{align*}
\begin{split}
S_t[V^0](x) = \sup_{u \in \U} \sup_{\mu \in \D} & \int_0^t \frac{1}{2} \xi(s)^\top D^{\mu(s)} \xi(s) 
\\
& - \frac{\gamma^2}{2} |u(s)|^2 \,ds + V^0\big(\xi(t)\big) \,.
\end{split}
\end{align*}

McEneaney showed in~\cite{mceneaney07} that $V(x)$ coincides with $\lim_{t \to +\infty} S_t[V^0](x)$ and that the latter limit is uniform on compact sets
if $V^0$ satisfies a quadratic growth condition (one requires that $\epsilon |x|^2 \leq V^0(x) \leq \lambda |x|^2$ for some positive constants $\epsilon, \lambda$ that are determined from the parameters).

We also associate to every value $\sigma\in \Sigma$
the semi-group $\{ S_t^\sigma \}_{t \geq 0}$ 
corresponding to the unswitched control problem obtained
by setting $\mu(s) \equiv \sigma$, i.e., 
\begin{align*}
\begin{split}
S_t^\sigma[V^0](x) = \sup_{u \in \U} & \int_0^t \frac{1}{2} \xi(s)^\top D^{\sigma} \xi(s) 
\\
& - \frac{\gamma^2}{2} |u(s)|^2 \,ds + V^0\big(\xi(t)\big) \, . 
\end{split}
\end{align*}
%

Computing
$S_\tau^\sigma[V^0]$ 
when $V^0(x) = x^\top P_0 x$,
reduces to solving the following indefinite Riccati differential equation,
 \begin{align*}
 \dot{P} = (A^\sigma)^\top P + PA^\sigma + PQ^\sigma P + D^\sigma \,,\quad P(0) = P_0 \,,
 \end{align*}
with $P(s)\in \Sn$. 
Indeed, 
we have $S_\tau^\sigma[V^0](x)=
x^\top P(\tau) x$.
We denote by $\ricc_{\tau,\sigma}$ the flow of this equation, so that
$\ricc_{\tau,\sigma} [P_0] \coloneqq P(\tau)$. 



\section{Minimal upper bounds in the Loewner order}
\label{sec:mub}

We begin by recalling some standard notation and definitions.
We denote by $\wp(X)$ the powerset of a set $X$.
We denote by 
$\Sn$ the space of symmetric matrices,
which
is equipped with the Frobenius scalar product defined by $\langle P, Q\rangle = \trace(PQ)$.
The product space $(\Sn)^p$ is equipped with the scalar product $ \langle (P_k)_k , (Q_k)_k \rangle = \sum \langle P_k, Q_k \rangle $.
The $n \times n$ identity matrix is denoted by $I_n$ and we use the shorthand $I_n^p$ to mean the $p$-tuple $(I_n, \dots, I_n)$.

 A symmetric matrix $\P$ is \emph{positive semidefinite} when the quadratic form $x^\top Px = \sum_{i,j} \P_{i,j}x_ix_j $ takes non-negative values for all vectors $x \in \R^n$, or equivalently when all the eigenvalues of $\P$ are non-negative. Then, we write $\P \succeq 0$. The set of positive semidefinite matrices is denoted by $\Snp$.
When $x^\top\P x$ is positive for all nonzero vectors $x \in \R^n$, we say that the matrix $\P$ is \emph{positive definite}.
The set of positive semidefinite matrices constitutes a convex cone in $\Sn$, meaning that $\lambda \P + \mu \Q \in \Snp$ for all $\P,\Q \in \Snp$ and non-negative $\lambda, \mu$. It is also closed and pointed ($\Snp \cap - \Snp = \{ 0 \}$), thus it defines an order relation on $\Sn$ by
\begin{align*}
\P \preceq \Q \iff \Q-\P \in \Snp \iff \Q-\P \succeq 0 \,.
\end{align*}
The partial order $\preceq$ is called the L\"owner order.

A classical result by Kadison~\cite{Kadison} shows that the set $\Sn$ equipped with this order constitutes an \emph{antilattice}, meaning that two matrices $\P,\Q \in \Sn$ have a supremum (least upper bound) if and only if they are comparable, meaning that $P\succeq Q$ or $Q\succeq P$.
If  the matrices $\P,\Q$ are not comparable, then they possess a continuum of \emph{minimal upper bounds}, i.e. upper bounds $\S$ such that $\P,\Q \preceq X \preceq \S$ implies $X = \S$, see~\cite{PAMS} for more information. 

Given a finite set of symmetric matrices $\Qs$, we denote by $\ub{\Qs} \coloneqq \{ X \in \Sn \colon X \succeq Q_i \,,Q_i \in \Qs \}$ the set of upper bounds of the matrices in $\Qs$.
This set is convex as an intersection of convex sets.
We also denote by $\mub{\Qs}\in \wp(\Sn)$ the subset of $\ub{\Qs}$ consisting of all minimal upper bounds of $\Qs$.
We use the symbol $\mub$ 
to denote a ``supremum'' operation which is {\em multivalued} owing to the antilattice character of $\Sn$.
The set $\mub{\Qs}$ coincides with the set of \emph{positively exposed points} $\ub{\Qs}$:
\begin{theorem}
\label{thm:mub_char}
The matrix $X$ is a minimal upper bound of a finite set of matrices $\Qs$ if and only if there is a positive definite matrix $C$ such that $X$ 
minimizes the map
\begin{align*}
Z \mapsto \langle C, Z \rangle
\end{align*}
 over the set $\ub{\Qs}$.
The minimizer, denoted by $X_C$, is unique, and, when the set $\Qs$ consists of two matrices $\P,\Q$, 
\begin{align}
X_C = \frac{\P+\Q}{2} + \frac{1}{2} C^{-1/2} \,\big| C^{1/2} (\P-\Q) C^{1/2} \big| \, C^{-1/2} \,.\label{e-def-xc}
\end{align}
Here, $C^{1/2}$ denotes the unique positive definite solution to the equation $X^2 = C$ and $|X| = (XX^\top)^{1/2}$.
\end{theorem}
\begin{skproof}
This is deduced from the optimality condition of the associated semidefinite program and from a generalization of the characterization of minimal upper bounds in~\cite[Theorem 3.1]{PAMS}.
When $C = I_n$, it can be checked that the matrix $X_C = P + |Q-P|$ satisfies the optimality conditions. The formula in the general case is obtained by a change of variable $X  \mapsto C^{1/2}XC^{1/2}$ and a symetrization in $P,Q$.
\end{skproof}

We say that the minimal upper bound $X_C$ is {\em selected} by the matrix $C$.
Note that
the expression of $X_C$ is similar to that of the maximum of two scalars: $\max(a,b) = ( a + b + |a-b| )/2$.
When more than two matrices are involved, 
minimal upper bounds can be computed by solving a semidefinite program.
By choosing $C = I_n$ in~\Cref{thm:mub_char}, 
we obtain a special minimal upper bound of $\Qs$,
denoted by  $ \sqcup_{\tr} \Qs$.

Finally, we point out a remarkable selection of a minimal upper bound of two positive semidefinite matrices. 
If all matrices in $Q_i \in \Qs$ are positive definite, a minimal upper bound is given by $(X^*)^{-1}$, where $X^*$ denotes the unique matrix that maximizes $\log \det X$ over all positive definite matrices $X$ such that $X \preceq Q_i^{-1}$ for all $i$. We denote this selection by $\sqcup_{\det} \Qs$.
This minimal upper bound corresponds to the minimum volume ellipsoid enclosing the ellipsoids $\{ x \in \R^n \colon x^\top Q_i^{-1}x \leq 1 \}$ and has thus received much attention across several fields, see~\cite{ACM} and references therein.
In particular, when the set $\Qs$ consists of two matrices $P,Q$, it has been shown in~\cite{ACM} that 
$\sqcup_{\det} \Qs$
can be obtained by selecting $C = P^{-1}$
(or $C=Q^{-1}$) in~\eqref{e-def-xc}.

\section{Tropical Kraus Maps}
\label{sec-kraus}
\subsection{Definitions}
\label{sec:trop_kraus_def}

In the sequel, we assume that we are given an index set $\M = \{ 1, \dots, m\}$, $m$ matrices $\A = \{ A_\sigma \}_{\sigma \in \Sigma}$.
We also assume that we are given a finite set $\W$ with $p$ elements and a map sending $\W \times \Sigma$ to $\W$, 
denoted by $(i,\sigma)\mapsto i\cdot \sigma $, for $i \in \W$ and $\sigma \in \Sigma$. We say that $(i,\sigma,j) \in \W \times \M \times \W$ is an {\em admissible transition} when $i \cdot \sigma = j$.

We now introduce the 
\emph{tropical Kraus map} $T$, 
from $(\Snp)^p$ to $(\wp(\Snp))^p$, 
whose $j$-th coordinate maps $X = (X_1,\dots,X_p) \in (\Snp)^p$ to
the subset of $\Snp$:
\begin{align*}
T_j (X) \coloneqq \mub \Big\{ A_\sigma^\top X_iA_\sigma \colon (i,\sigma) \in \W \times \M\,, i \cdot \sigma = j \Big\} \,.
\end{align*}


\begin{remark}
In the setting of quantum information~\cite{kraus}, a Kraus map is given by $X \mapsto \sum_i A_iXA_i^\dag $ and acts on the set of density matrices (positive semidefinite matrices of trace $1$).
We say that the map $T$ is a \emph{tropical Kraus map} since it ressembles the latter, except the sum has been replaced with the \enquote{supremum} operation $\bigvee$, and the matrices have been transposed. The transposition is not surprising: classical Kraus maps provide a forward propagation of density matrices, whereas we are interested in Lyapunov functions, whose propagation follows a backward scheme. I.e., the present tropical Kraus maps are analogues to the {\em adjoints} of classical Kraus maps. Finally, as we work with real quadratic forms,
instead of hermitian forms, the hermitian conjugate $\dag$ is replaced by transposition $^\top$.
\end{remark}

We also consider a variant of the tropical Kraus map, adapted to the optimal control problem in~\Cref{sec:clas_hjb}, defining the map $M_{\tau}$ from
$(\Snp)^p$ to $(\wp(\Snp))^p$ by: 
\begin{align*}
(M_{\tau})_j (X) \coloneqq \mub \Big\{ \ricc_{\tau, \sigma} X_i \colon (i,a) \in \W \times \M\,, i \cdot \sigma = j \Big\} \,.
\end{align*}


\iftrop

\begin{remark}
\label{rem:trop_rem}
A tropical Kraus map from $(\Snp)^p$ to $(\wp(\Snp))^p$ can be represented
by a new tropical Kraus map from $\Snpp$ to $\wp(\Snpp)$, preserving the set
of block-diagonal matrices.
Indeed, for $X = \diag( X_1, \dots, X_m)$, these maps can be written as
\begin{align*}
\begin{split}
\mub \Big\{ f_\alpha(X) &\colon \alpha =   (i,\sigma,j) \;, i \cdot \sigma = j \Big\} \\
&\text{with}\;f_\alpha(X) = A_\alpha^\top  X A_\alpha\,.
\end{split}
\end{align*}
where $A_\alpha = E_{ij} \otimes A_\sigma$ denotes the lifted set of matrices.
Here, $E_{ij}$ denotes the matrix with $1$ in the $(i,j)$-th entry
and $0$ everywhere else and $\otimes$ is the Kronecker product.
Indeed, whenever the selection matrix $C$ in~\Cref{thm:mub_char} is block-diagonal $C = (C_1, \dots, C_m)$, the associated minimal upper bound is also block diagonal and the value of the $j$-th block is exactly the minimal upper bound in $T_j(X)$ that is selected by $C_j$.
The case where the matrix $C$ is not block-diagonal does not appear in our analysis.
The same remark applies to the variant $M_\tau$ of the tropical Kraus map.
\end{remark}
\fi

\subsection{Non linear eigenvalue and fixed points problems associated to Tropical Kraus Maps}
%

The tropical Kraus map $T$ is positively homogeneous, meaning that $T(\alpha X) = \alpha T(X)$ for all $X \in \Snp$ and $\alpha \geq 0$. This suggests to consider a multivalued eigenproblem. A \emph{(non-linear) eigenvector} of $T$,
associated to the \emph{eigenvalue} $\lambda$ is a nonzero matrix $X\in \Snp$ such that $\lambda X_j \in T_j(X)$ holds for
all $j\in \W$. We write $\lambda X\in T(X)$ for brevity. This notation
is licit since we can identify $T(X)$ which is an element of $(\wp(\Sn))^p$ to an element of $\wp\big( (\Sn)^p \big)$.  


The following result shows that a non-linear eigenvalue of the tropical Kraus
map provides an upper bound for the joint spectral radius. 
\begin{theorem}
\label{thm:fix_jsr}
If the multivalued eigenvector problem $\lambda X \in \T(X)$ has a solution such that the matrix $\sum_{j\in \W} X_j$ is positive definite, then, the map 
\[
v(z) \coloneqq \sup_{j\in \W} (z^\top  X_j z)^{1/2}
\]
is a norm, and $v(A_\sigma z) \leq \sqrt{\lambda} v(z)$
holds for all $z\in \R^n$ and $\sigma\in \Sigma$. 
In particular, the joint spectral radius of $\A$ 
does not exceed $\sqrt{\lambda}$.
\end{theorem}
\begin{proof}
The proof of this result is similar to the proof Theorem 2.4 by Ahmadi et al.~\cite{pathcomplete}. Indeed, a non-linear eigenvector $X$ of the tropical Kraus map $T$ provides a feasible point of the semidefinite program considered in~\cite[Theorem 2.4]{pathcomplete}.
\end{proof}

We have the following analogous result for the switched linear quadratic
control problem.
\begin{theorem}
\label{thm:fix_hjb}
If the multivalued fixed point problem $X\in M_\tau(X)$ has a 
solution,
then the map $V \coloneqq z \mapsto \sup_i z^\top X_iz$ determines a 
sub-invariant function
of the Lax-Oleinik semi-group $S_t$, meaning that:
\begin{align*}
\max_\sigma S_t^\sigma[V](z) \leq V(z)\;\text{for all} \; z\,.
\end{align*}
\end{theorem}

\subsection{Existence of non-linear eigenvectors of tropical Kraus maps and computation by a Krasnoselskii-Mann iteration}

\label{sec:trop_iter}

For a completely positive map, $X\mapsto \sum_i A_i XA_i^\top $,  the existence
of a positive semidefinite eigenvector follows from the Perron-Frobenius
theorem~\cite{lemmensnussbaum}. Moreover, such an eigenvector is necessarily positive
definite as soon as the map is {\em irreducible} in the Perron-Frobenius sense, meaning that the map does not
leave invariant a non-trivial face of the closed cone $\Snp$.
As shown in~\cite{Farenick96}, the latter condition holds if and only
if the set of matrices $\{A_i\}$ is {\em irreducible} in the
algebraic sense, meaning that there is no non-trivial subspace
invariant by each matrix in this set.

In order to show that tropical Kraus maps have eigenvectors, we specialize the multivalued map $T$ defined in~\Cref{sec:trop_kraus_def} 
by fixing a selection of minimal upper bound $\sqcup$. We obtain the map $\Tsel$ defined on $(\Snp)^p$ by
\begin{align*}
\Tsel_j(X) \coloneqq
 \bigsqcup
  \Big\{ A_\sigma ^\top  X_i A_\sigma \colon i \cdot \sigma = j \Big\} \,.
\end{align*}
We will prove that the map $\Tsel$ 
has a non-linear eigenvector if the selection $\sqcup$ is the "minimum volume" selection $\sqcup_{\det}$.
We introduce the ``non-commutative simplex'' $\Delta_p \coloneqq \{ X \in (\Snp)^p \colon  \langle I_n^p ,X \rangle= 1 \}$ and the map $\widehat{\Tsel}$ sending $\Delta_p$ to itself:
\begin{align*}
\widehat{\Tsel}(X) \coloneqq \frac{1}{2} \Big[ \frac{1}{\langle I_n^p,\Tsel(X)\rangle } \Tsel(X) + X \Big] \,.
\end{align*}
Observe that, independently of the selection $\sqcup$, a fixed point $X \in \Delta_p$ of the map $\widehat{\Tsel}$ yields an eigenvector for the map $\Tsel$ associated with the eigenvalue $\langle I_n^p,\Tsel(X)\rangle$.
We can now state the theorem.

\begin{theorem}\label{th-tsel}
If the set of matrices $\{ E_{ij} \otimes A_\sigma \colon i \cdot \sigma = j\}$ is irreducible,
then the map $\widehat{\Tsel}$ 
has a positive definite fixed point.
\end{theorem}
\begin{skproof}
The proof relies on the inequality $p^{-1} \sum_{1\leq k\leq p} Q_k \preceq \sqcup_{\det} Q_k \preceq \sum_{1\leq k\leq p} Q_k$
which can be deduced from~\cite[Theorem 4.1]{ACM} when the set $\A$ is reduced to two matrices.
The irreducibility of $\A$ implies that there is an integer $q$ such that the map $\widehat{\Tsel}$ iterated $q$ times sends $\Delta_p$ into its interior. 
We then show the existence of a convex compact set $K$ included in the interior of $\Delta$ which is invariant by $\widehat{\Tsel}$. 
The operator $(Q_1,\dots,Q_p) \mapsto \sqcup_{\det} \{Q_1,\dots,Q_p\}$
is continuous
on the interior of $\Delta$, hence, $\widehat{\Tsel}$ is continuous
on $K$.  We conclude by applying Brouwer's fixed-point theorem
to $\widehat{\Tsel}$. 
\end{skproof}
We obtain as an immediate corollary:
\begin{corollary}\label{cor-eigen}
If the set of matrices $\{ E_{ij} \otimes A_\sigma \colon i \cdot \sigma = j\}$ is irreducible, then, the tropical
Kraus map $T$ has a positive definite eigenvector.
\end{corollary}
\begin{remark}
Several basic methods allow one to prove non-linear extensions of the Perron-Frobenius
theorem. These involve contraction properties with respect to Hilbert's projective metric, Brouwer fixed point theorem, or monotonicity properties,
see~\cite{lemmensnussbaum}. A direct application of all these methods %
fails in the case of the tropical Kraus maps $\Tsel$, which are not contracting,
not monotone, and which do not have continuous extensions
to the closure of the cone on which they act. This is why the proof of 
\Cref{cor-eigen} relies on the detour through $\widehat{\Tsel}$ in \Cref{th-tsel}.
\end{remark}

In order to compute a fixed point of the map $\widehat{\Tsel}$, we compute successive iterates starting from a positive definite matrix $X^{(0)}$. 
This yields a Krasnoselskii-Mann-type scheme:
\begin{align*}
X^{(k+1)} =  \frac{1}{2} \Big[ \frac{1}{\langle I_n^p,\Tsel(X^{(k)})\rangle } \Tsel(X^{(k)}) + X^{(k)} \Big] \,.
\end{align*}
This is a power-type iteration, involving a renormalization
and a ``damping term'' (addition of $X^{(k)}$) to avoid oscillations.
This should be compared with the classical Krasnoselskii-Mann iteration,
which applies to non-expansive mappings $T$, and takes
the form $X^{(k+1)}=(T(X^{(k)})+X^{(k)})/2$, see~\cite{kmconvergence}.

\begin{remark}
There is a multiplicative variant of the iteration, defined by
\begin{align*}
X^{(k+1)} =  \Big[ \frac{\Tsel(X^{(k)})}{\langle I_n^p,\Tsel(X^{(k)})\rangle }  \Big] \# X^{(k)}\,,
\end{align*}
where $P \# Q \coloneqq P^{1/2}\big( P^{-1/2} Q P^{-1/2} \big)^{1/2} P^{1/2}$ denotes the \emph{Riemannian barycenter} of the positive definite matrices $P,Q$, see~\cite[Chapter 2]{bhatiaPDM} for more information. 
We can show that this multiplicative version does converge in the ``commutative case'', i.e., when $n = 1$. Then, the map $X \mapsto {\langle I_n^p,\Tsel(X)\rangle }^{-1}{\Tsel(X)}$ is nonexpansive in the Hilbert metric~\cite{lemmensnussbaum}, and then, the general result of~\cite{kmconvergence} can be applied. The additive version can also be shown to be converging when $n=1$, by a reduction to the same result, but the proof is more involved. 
\end{remark}

We use a different iteration scheme to compute fixed points of the variant $M_\tau$. 
First, we specialize again the multivalued map $M_\tau$ with a minimal upper bound selection $\sqcup$ to obtain the map $M_\tau^{\text{sel}}$. Then, we compute iteratively
\begin{align}
\label{eq:mc_iter}
X^{(k+1)} =  M_\tau^{\text{sel}}(X^{(k)})  \,.
\end{align}
We can show that this iteration converges on some "good" instances of the problem, when we choose $\sqcup \coloneqq \sqcup_{\det}$.
Indeed, 
the (indefinite) Riccati flow is a contraction in the Thompson metric, with contraction rate $\alpha > 0$ 
determined by the parameters of the flow~\cite[Corollary 4.7]{gaubertqu}.
Moreover, as stated by Allamigeon et al.~\cite{ACM}, the selection $\sqcup_{\det}$ has a Lipschitz constant in the Thompson metric which is not larger than $1+(4/\pi) \log n$ (this bound is conservative). Combining these results, the iteration in~\Cref{eq:mc_iter} is guaranteed to converge locally when $\exp(\alpha \tau) > 1+(4/\pi) \log n$.
The contraction rate $\alpha$ depends on an interval $\{ X \in \Sn \colon \lambda_1 I_n \preceq X \preceq \lambda_2 I_n \}$ and this interval is not preserved by $\sqcup_{\det}$, hence we do not have global convergence. Determining whether there is a minimal upper bound selection that preserves this interval and that has a finite Lipschitz constant in Thompson's metric, to be used instead of $\sqcup_{\det}$ remains an open problem.


\subsection{Implementation issues}

We describe in this section the resolution to several issues that arise in the implementation of the iterative scheme.

First, 
in the iterative scheme to approximate the joint spectral radius, we introduce a small positive perturbation $\varepsilon$ in the computation:
\begin{align*}
\Tsel_j(X) \coloneqq
 \bigsqcup
  \Big\{ A_\sigma ^\top  X_i A_\sigma + \varepsilon I_n \colon i \cdot \sigma = j \Big\} \,.
  \end{align*}
In practice, we use values for $\varepsilon$ in the range $10^{-4}-10^{-2}$. This additional parameter allows us to obtain,
in a finite number of iterations,
a solution $(X,\rho)$ that satisfies $\rho^2 X_j \succeq A_\sigma^\top  X_i A_\sigma$ for all admissible $(i,\sigma,j)$.
Moreover, this parameter absorbs numerical imprecisions that may appear during the computation and ensures that the matrices $X_j$ are positive definite, so the assumptions of~\Cref{thm:fix_jsr} and~\Cref{thm:fix_hjb} are satisfied. 

We choose the selection $\sqcup \coloneqq \sqcup_{\text{tr}}$. Then,
when the set $\Sigma$ contains only two elements, $\Tsel_j(X)$ can be computed
analytically thanks to~\Cref{thm:mub_char}. When $\Sigma$ has more than
two elements,
instead of computing the true minimal upper bound $\sqcup_{\tr} \Qs$, we compute an approximation by sequential evaluation: $Q_1 \sqcup_{\tr} ( Q_2 \sqcup_{\tr} ( \dots \sqcup_{\tr} Q_p ))$. 

Finally, as pointed out in~\cite{qucdc}, the propagation of the Riccati operator $\ricc_{\tau,\sigma}$ on a single quadratic form $P_0$ is computed analytically by $\ricc_{\tau,\sigma} = Y(\tau)X(\tau)^{-1}$, with $\mathcal{M}^\sigma =  \begin{psmallmatrix}
-A^\sigma & -Q^\sigma \\
D^\sigma & A^\sigma
\end{psmallmatrix}$
and
$
\big(X(\tau) \,;\, Y(\tau)\big)^\top 
= \exp \big( \mathcal{M}^\sigma \tau)
\big(
I_n \,;\,
P_0
\big)^\top 
$.

\section{Experimental results}\label{sec-exp}

\subsection{Path-complete graph Lyapunov functions}
\label{sec:pathcomplete}

In~\cite{pathcomplete}, Ahmadi and al.\ developed a method to compute an overapproximation of the joint spectral radius of a finite set of matrices,
to which we shall compare our method.

Given a set of states $\W$ and an alphabet $\Sigma$, an \emph{edge} of a \emph{labeled graph} is a triple $(i,\sigma,j) \in \W \times \Sigma \times \W$. The set of edges is denoted $E$. Such a graph is called \emph{path-complete} if for every state $i$ and letter $\sigma$, there is some state $j$ such that $(i,\sigma,j)$ is an edge.

Let $\A = \{ A_\sigma\}_{\sigma \in \Sigma}$ denote a finite set of $n \times n$ matrices and $\rho$ a non-negative real number.
In~\cite{pathcomplete}, the authors examine graphs, denoted $\G(X,\rho)$, whose states are positive definite matrices $\{X_i\}_i$ and whose edges are determined by
\begin{align*}
(i,\sigma,j) \in E \iff A_\sigma^\top  X_i A_\sigma \preceq \rho^2 X_j \,.
\end{align*}

The main theorem in~\cite{pathcomplete} shows that the construction of a path-complete graph $\G(X,\rho)$ gives an upper bound of the joint spectral radius:
\begin{theorem}[Theorem 2.4~\cite{pathcomplete}]
If the graph $\G(X,\rho)$ is path-complete for some set of positive definite matrices $\{X_i\}_i$, then $\rho(\A) \leq \rho$.
Moreover, the map $V \colon z \mapsto \max_i z^\top X_i z$ is a Lyapunov-type function: it satisfies $V(A_\sigma x) \leq \rho^2 V(x)$ for all $\sigma \in \Sigma$ and $x \in \R^n$.
\end{theorem}

In practice, for a fixed value of $\rho$ and a given path-complete graph $\G$, checking the existence of a path-complete graph $\G(X,\rho)$ whose edges coincide with $\G$ amounts to checking the feasibility of an LMI.
A bisection scheme is then implemented to refine $\rho$.
For brevity, we shall refer to this method as the LMI method.

A class of graphs which provides good theoretical and experimental approximations is the class of \emph{De Bruijn} graphs. The set of states of the De Bruijn graph of order $d$ is the set $\Sigma^d$ of words built on $\Sigma$ which have length $d$. There is an edge $(i,\sigma,j)$ between states $i$ and $j$ if and only if $i = \sigma_1\dots \sigma_d$ and $j = \sigma_2\dots \sigma_d \sigma$. This graph, denoted by $D_d$, is path-complete by construction.

\subsection{McEneaney's curse of dimensionality attenuation scheme}

We assume that $V^0$ is a quadratic function $V^0(x) = x^\top P^0x$.
The method of~\cite{mceneaney07} that solves the linear quadratic optimal control problem described in~\Cref{sec:clas_hjb} approximates
the value function $V$ 
by a finite supremum of quadratic forms
\begin{align}\label{eq-bill}
V \approx \sup_{\sigma_1,\dots,\sigma_N \in \Sigma}
 S_\tau^{\sigma_1} \cdots S_\tau^{\sigma_N}[V^0] \,,
\end{align}
where $\tau$ is a (small) time discretization step and $N$ is a maximal
number of switches. The latter supremum represents the value of a modified optimal control problem, in horizon
$\tau N$, in which switches occur only at times multiple
of $\tau$. 
We have $S_\tau^{\sigma_1} \cdots S_\tau^{\sigma_N}[V^0](x)
= x^\top  Q x$, where $Q= \ricc_{\tau,\sigma_1}\circ\dots \circ \ricc_{\tau,\sigma_N} (P_0)$, can be computed by integrating successive Riccati equations, which allows
us to evaluate the expression in~\Cref{eq-bill}.

The propagation of a quadratic form by the Lax-Oleinik semi-group has only a cubic cost in terms of the dimension $n$, contrary to classical grid-based methods whose cost is exponential in the dimension.
In this sense, the curse of dimensionality has been reduced.
 However, the memory footprint of this method is exponential in the number of switches, since $m^N$ quadratic forms are computed after $N$ iterations.
Several pruning schemes have been proposed in~\cite{qucdc} to limit this growth.
  This is a costly operation, indeed,
$99\%$ of the computation time is spent solving LMIs inside the pruning procedure~\cite{qucdc}.

\subsection{Application to the joint spectral radius}

Given that the approximation of the joint spectral radius $\rho(\A)$ depends on the graph $\G$ that underlies the analysis, we denote by $\rhot(\A, \G)$ the approximation obtained as (the square root of) an eigenvalue of a tropical Kraus map and by $\rho(\A,\G)$ the one obtained by solving LMIs.

The map $\cdot$ sending $\W \times \Sigma$ to $\W$ defined in~\Cref{sec:trop_kraus_def} can be interpreted as a path-complete graph. For this reason, our method, when applied to the joint spectral radius, is a relaxation of the path-complete Lyapunov function framework, and thus we always have
\begin{align*}
\rho(\A) \leq \rho(\A,\G) \leq \rhot(\A,\G) \,.
\end{align*}

However, we shall see that the tropical method is much more tractable, so we may use a bigger graph and sometimes get a better approximation than by solving LMIs.

We compare the performance of our algorithm with the path-complete graph Lyapunov method, in terms of computation time and accuracy of the approximation of the joint spectral radius, measured by
\begin{align*}
\delta \coloneqq (\rhot(\A,\G') - \rho(\A,\G) / \rho(\A,\G) \,.
\end{align*}

All the experiments were implemented in Matlab, running on one core of a $2.2$ GHz Intel Core i7 with $8$ GB RAM. The semidefinite programs were solved using YALMIP (R$20160930$), calling SeDuMi $1.3$.

\subsubsection{Accuracy of the approximation}


We generate $600$ pairs $\A = \{A_1, A_2\}$ of random $6 \times 6$ matrices. For each of these pairs, we compare the approximation of the joint spectral radius obtained by the LMI method on the graph $D_3$  (involving $8$ positive semidefinite matrices) and by the tropical Kraus method on the graph $D_6$ (involving $64$ positive semidefinite matrices).
On these examples, 
we report that the tropical method obtains 
a similar approximation of the joint spectral radius, within a margin of $2.5\%$, and outperforms the LMI-method on $25\%$ of these examples.
Moreover, whereas the LMI-method requires between $3$s and $5$s to obtain this approximation, the tropical method consistently returns an approximation in $1$s.

\subsubsection{Scalability - dimension}

We generate random pairs of $n \times n$ matrices, for $n$ ranging from $5$ to $500$. We use again the De Bruijn graph $D_3$ in the LMI-method and the graph $D_6$ in the tropical Kraus method. We show in~\Cref{tab:scal} the mean computation time required to obtain an overapproximation and the mean relative accuracy of the tropical method with respect to the LMI-method, when it applies.
First, one can observe the major speedup provided by the tropical method, from $4$ times faster when $n = 5$ to $80$ times faster for $n = 40$.

Also note that the tropical method is using $8$ times more quadratic forms in its analysis and remains much faster than the LMI-method. Thus, given a fixed time budget, the tropical method enjoys more flexibility regarding the size of the graph that is used in the analysis.

Moreover, observe that the LMI-method cannot provide estimates on the joint spectral radius for values of $n$ greater than $45$, whereas the tropical method easily reaches values of $n$ greater than $100$.

Finally, the accuracy of the tropical approximation remains within a $3\%$ margin of the one obtained by the LMI-method.


\begin{table}[t]
\caption{Comparison of the methods with respect to the size of the dimension of the matrices.}
\label{tab:scal}
\begin{center}
\begin{tabular}{c|c|c|c|c}
\begin{tabular}{c}
Dimension \\ $n$
\end{tabular} &
\begin{tabular}{c}
CPU time \\ (tropical)
\end{tabular} &
\begin{tabular}{c}
CPU time \\ (LMI)
\end{tabular} &
\begin{tabular}{c}
Upper bound \\  on  $\rho(\A)$  \\(tropical)
\end{tabular} &
\begin{tabular}{c}
Upper bound \\  on   $\rho(\A)$ \\ (LMI) 
\end{tabular}
\\
\hline 
$5$ &
$0.9$ s& 
$3.1$ s&
$2.767$ &
$2.7627$
\\
\hline 
$10$ &
$1.5$ s&
$4.2$ s&
$3.797$ &
$3.7426$ 
\\ 
\hline
$20$ &
$3.5$ s&
$31$ s&
$5.4093$ &
$5.3891$ 
\\ 
\hline
$30$ &
$7.9$ s&
$3$min &
$6.2038$ &
$6.1942$ 
\\ 
\hline
$40$ &
$13.7$ s&
$18$min & 
$7.3402$ &
$7.3363$ 
\\ 
\hline
$45$ &
$18.1$ s&
$-$ &
$7.687$ &
$-$ 
\\ 
\hline
$50$ &
$25.2$ s&
$-$ &
$8.1591$ &
$-$ 
\\ 
\hline
$100$ &
$1$min &
$-$ &
$11.487$ &
$-$ 
\\ 
\hline
$500$ &
$8$min &
$-$ &
$25.44$ &
$-$ 
\end{tabular}
\end{center}
\end{table}

\subsubsection{Scalability - graphs}

We now analyze the influence of the order of the De Bruijn graph $D_d$ used in the analysis on the computation of the upper bound on the joint spectral radius obtained by both methods. 
We use the matrices
$
A_1 = \begin{psmallmatrix*}[r]
-1 & 1 & -1 \\
-1 & -1 & 1 \\
 0 & 1 & 1
\end{psmallmatrix*}
$ and $
A_2 = \begin{psmallmatrix*}[r]
-1 & 1 & -1 \\
-1 & -1 & 0 \\
 1 & 1 & 1
\end{psmallmatrix*}$, introduced in~\cite{Guglielmi2014}. Their joint spectral radius is $\rho(\A) = 1.78893$.

We show in~\Cref{tab:scal_g} the upper bound on the joint spectral radius and the computation time with respect to the length order $d$ of the De Bruijn graphs.

\begin{table}
\caption{Comparison of the methods w.r.t. the size of the graph $D_d$.}
\label{tab:scal_g}
\begin{center}
\begin{tabular}{c|c|c|c|c|c}
Order $d$ &
$2$ &
$4$ &
$6$ &
$8$ &
$10$ \\
\hline
Size of $\W$ &
$8$ &
$32$ &
$128$ &
$512$ &
$2048$ \\
\hline
\begin{tabular}{c}
CPU time  \\
 (tropical) 
\end{tabular} &
$0.03$s &
$0.07$s &
$0.4$s &
$2.0$s &
$9.0$s \\
\hline
\begin{tabular}{c}
CPU time  \\
 (LMI) 
\end{tabular} &
$1.9$s &
$4.0$s &
$24$s &
$1$min &
$10$min \\
\hline
\begin{tabular}{c}
Upper bound on  \\
$\rho(\A)$ (tropical) 
\end{tabular} &
$1.842$ &
$1.821$ &
$1.804$ &
$1.800$ &
$1.801$ \\
\hline
\begin{tabular}{c}
Upper bound on  \\
$\rho(\A)$ (LMI) 
\end{tabular} &
$ 1.8216$ &
$ 1.7974$ &
$ 1.7957$ &
$ 1.7922$ &
$ 1.7905$ \\
\end{tabular}
\end{center}
\end{table}

\subsection{A faster curse of dimensionality attenuation scheme}

We now apply the iteration scheme described in~\Cref{sec:trop_iter} to the approximation of the value function $V$.
In all examples, we measure the quality of the approximation of the value function as in~\cite{mceneaney09, qucdc} with the H-infinity back-substitution error $\max_{x^\top x \leq 1} | H(x, \nabla V(x)) |$ 
on the subspace spanned by the canonical vectors $e_1$ and $e_2$.

The first example is Example $1$ in~\cite{mceneaney07} and we use the instance of~\cite{qucdc} in the second example.  Examples $3$ and $4$ are randomly generated examples that satisfy the technical assumptions in~\cite{mceneaney07}.

\Cref{tab:mc} depicts the results of the computations. In particular, we give the backsubstitution error at the beginning of the computation, when the value function is approximated by a single quadratic form ($Q(x) = 0.1 |x|^2$ in all cases) and the final backsubstitution error when the scheme has converged.

\begin{table}[t]
\caption{Numerical benchmarks of the tropical Kraus method applied to McEneaney's switched linear quadratic problem}
\label{tab:mc}
\begin{center}
\begin{tabular}{c|c|c|c|c|c}
Example &
$1$ &
$2$ &
$2$ &
$3$ &
$4$ \\
\hline
Dimension &
$2$ &
$6$ &
$6$ &
$20$ &
$20$ \\
\hline
Size of $\Sigma$ &
$3$ &
$6$ &
$6$ &
$2$ &
$4$ \\
\hline
$\tau$ &
$0.05$ s &
$0.2$ &
$0.1$ &
$0.1$ &
$0.1$ \\
\hline
Size of $\W$ &
$81$ &
$216$ &
$1296$ &
$128$ &
$256$ \\
\hline
Initial error &
$0.78$ &
$1.12$ &
$1.12$ &
$4.2$&
$4.79$\\
\hline
Final error &
$0.047$ &
$0.071$ &
$0.090$ &
$0.0006$ &
$0.17$\\
\hline
Iterations &
$194$ &
$115$ &
$200$ &
$55$ &
$288$\\
\hline
CPU time &
$8$ s&
$41$ s&
$5$ min&
$5$ s &
$2.5$ min
\end{tabular}
\end{center}
\end{table}

\addtolength{\textheight}{-4cm}   

\section{CONCLUDING REMARKS}
We introduced a new method to approximate the value function of optimal control problems for switched systems. This method applies to situations in which the 
evolution semi-group of the unswitched problem preserves the space of quadratic forms. This includes the computation of the joint spectral radius
and a class of linear quadratic control problems with switches considered
by McEneaney. Our scheme belongs to the family of max-plus methods as it
approximates the value function by a supremum
of quadratic forms. It avoids the recourse to semidefinite programming
(which was the bottleneck of earlier max-plus methods) by a reduction
to a non-linear eigenproblem, exploiting the geometry of the L\"owner
order. This leads to a major speedup, allowing us to obtain approximate solutions of instances in dimension up to $100$ in the case of the joint spectral radius, and 20 for McEneaney's problem, hardly accessible by other methods.

Let us now point out the limitations of the present approach, together
with possible ways to overcome them. 

A key ingredient in
our method is the replacement of LMI formulations by a selection of
minimal upper bounds in the L\"owner order. This induces a 
``relaxation gap'', which is difficult to estimate as it depends
on the specific selection which is used. Another difficulty is
that these selections may be expansive, resulting in a potential
instability or lack of convergence of the iterative scheme. In the case of the joint spectral radius, experiments indicate that the scheme does converge
(although a proof of convergence in the general case is missing). In the case of switched linear quadratic control problems, we do have a proof of convergence
in the discrete time case. This proof requires the contraction rates of the Riccati flows arising in our problem to be sufficiently small to absorb the expansiveness of the selection of a joint. In practice, the scheme converges in
more general circumstances. However, in the continuous time case, the precision of the scheme becomes limited as it can be unstable for small values of the time discretization step. In other words, the scheme currently allows one to compute
quickly a coarse approximation of the solution of a Hamilton-Jacobi PDE.
The most promising improvement of the scheme may be to
adapt dynamically the selection of a minimal upper bound, which
will reduce the relaxation gap, and might also improve the convergence.
This is left for further work.

\section*{Acknowledgements}

We thank the reviewers for their helpful comments.




\bibliographystyle{plain}
\bibliography{cdc17,bibliozheng}

\end{document}